\documentclass[12pt]{article}
   \usepackage{e-jc}

\let\theoremstyle\nothing

\usepackage{amsthm, amsmath, amssymb, pstricks}
\usepackage{graphicx}

\newtheorem{theorem}{Theorem}[section]
\newtheorem{lemma}[theorem]{Lemma}
\newtheorem{corollary}[theorem]{Corollary}
\theoremstyle{definition}
\newtheorem{definition}[theorem]{Definition}

\newtheorem{proposition}[theorem]{Proposition}
\theoremstyle{remark}
\newtheorem{remark}[theorem]{Remark}

\makeatletter
\@addtoreset{equation}{section}
\makeatother

\newcommand\R{\mathbb R}   
\newcommand\Po{\mathcal P}   
\newcommand\Co{\mathcal C}   

\newcommand\Lo{\mathcal L}
\newcommand\No{\mathcal N}
\newcommand\Mo{\mathcal M}

\DeclareMathOperator{\conv}{conv}   

\def\R{\mathbb{R}}

\def\<{\left<}
\def\>{\right>}

\def\A{\mathcal{A}}
\def\B{{\mathcal{B}}}

\def\H{\Delta}

\def\G{\Gamma}

\def\des{\mathrm{des}}

\def\I{{\mathcal{I}}}

\newcommand{\set}[1]{ \left\{ #1 \right\} }
\def\mm{{\mathcal M}}

\def\A{{\mathcal A}}

\def\P{{\mathcal P}}

\def\A{{\mathcal Aff}}

\def\conv{{\mathrm{conv}}}

\renewcommand{\conv}{\mathrm{conv}}

\renewcommand{\<}{\left\langle}
\renewcommand{\set}[1]{\left\lbrace #1 \right\rbrace} 
\renewcommand{\>}{\right\rangle}

\newcommand{\inD}[1][\relax]{\def\argone{#1}\def\temprelax{\relax}
  \ifx\argone\temprelax\right.\else\,\middle|#1\right.{}\fi}

\title{Lattice Path Matroid Polytopes}
\begin{document}
\date{}
\maketitle
\begin{center}
\author{Hoda Bidkhori\\

\small  Massachusetts Institute of Technology, Cambridge, MA, 02139,\\[-0.8ex]
\small  \texttt{bidkhori@mit.edu}}
\end{center}
\and
\begin{center}
\end{center}    
\begin{abstract}

Fix two lattice paths  $P$ and $Q$  
from $(0,0)$ to  $(m,r)$ that use East and North steps with $P $ never going above $Q$. Bonin et al. in~\cite{B1} show that the lattice paths that go from $(0,0)$ to $(m,r)$ and remain bounded by $P$ and $Q$ can be identified with the bases of a particular type of transversal matroid, which we call it a lattice path matroid.

In this paper, we consider properties of  lattice path matroid polytopes. These are the polytopes associated to the lattice path matroids. We investigate their face structure, decomposition, triangulation,  Ehrhart polynomial and volume.

\noindent {Keywords}: 
\end{abstract}


\section{Introduction}

In this paper we discuss  a special class of matroid polytopes which we call the Lattice path matroid polytopes. With every pair of lattice paths P and Q that have  a common
endpoints we associate a matroid in such a way that the bases of the matroid
correspond to the paths that remain in the region bounded by $P$ and $Q$. These matroids,
which we call lattice path matroids, appear to have a wealth of interesting  and striking properties.

For any matroid one can associate a matroid polytope by taking the convex hull
of the incidence vectors of the bases of the matroid. The last few years has seen
a flurry of research activities around matroid polytopes, in part because their
combinatorial properties provide key insights into matroids and in part because
they form an intriguing and seemingly fundamental class of polytopes which
exhibit interesting geometric features. The theory of matroid polytopes has
gained prominence due to its applications in algebraic geometry, combinatorial
optimization, Coxeter group theory, and most recently, tropical geometry. In
general matroid polytopes are not well understood.

In this paper we investigate properties of the lattice path matroid polytopes which are the polytopes associated to the lattice path matroids. This class of matroid polytopes have many interesting properties and they are belong to  important class of polytopes such as Alcoved Polytopes, Generalized Permutahedron, Polypositroids ~\cite{AP}, ~\cite{AP2}. This Polytope is also closely related to Stanley-Pitman Polytopes discussed by Stanley ~\cite{PiSt}.

 The combinatorial and structural properties of  the Lattice Path Matroids are studies by Bonin et. al. in 
 ~\cite{B1} and ~\cite{B2}. In this paper we discover the face structure, decomposition, triangulation,  Ehrhart polynomial and volume of the lattice path matroid polytopes.
\section{Definitions and Background} \label{background}
A \emph{Matroid} $\Mo$ is a finite collection of subsets
$\mathcal{F}$ of $[n] = \{1,2,\ldots,n\}$ called \emph{independent
sets} such that the following properties are satisfied:
\begin{enumerate}
\item $\emptyset \in \mathcal{F}$
\item  If $U \in
\mathcal{F}$ and $V \subseteq U$ then $V \in \mathcal{F}$ 
\item If $U, V \in \mathcal{F}$ and $|U| = |V| + 1$ there exists $x \in U
\setminus V$ such that $V \cup x \in \mathcal{F}$
\end{enumerate}

 \emph{Bases} are defined  to be maximal independent sets of a matroid.  Let $\mathcal{B}$
be the set of bases of a matroid $\Mo$.
 If $B =
\{\sigma_1,\ldots,\sigma_r\} \in \mathcal{B}$, the
\emph{incidence vector of B} is defined as $  e_B := \sum_{i=1}^r 
e_{\sigma_i}$, where $ e_j$ is the standard elementary $j$th
vector in $\mathbb{R}^n$. We define \emph{matroid polytope} of $\Mo$  as
$\Po(\Mo) := \conv \{\,  e_B  \mid B \in \mathcal{B} \, \}$, where
$\conv(\cdot)$ denotes the convex hull.

\vspace{2mm}

The set system $\mathcal{A}=\{A_j: j\in J\}$  is a multiset
of subsets of a finite set~$S$. A \emph{transversal} of $\mathcal{A}$ is a set $\{x_j:j\in J\}$
of $|J|$ distinct elements such that $x_j\in A_j$ for all $j$ in
$J$. A \emph{partial transversal} of $\mathcal{A}$ is a transversal
of a set system of the form $\{A_k:k \in K\}$ with $K$ a subset of $J$.

\

Edmonds and Fulkerson showed the following fundamental result:
\begin{theorem}
The  partial transversals of a set system $\mathcal{A}=\{A_j: j\in
J\}$ are the independent sets of a matroid on $S$.
\end{theorem}
\
A \emph{transversal matroid} is a matroid whose independent sets are
the partial transversals of some set system $\mathcal{A}=\{A_j: j\in
J\}$; we say that $\mathcal{A}$ is a \emph{presentation} of the
transversal matroid. The bases of a transversal matroid are the
maximal partial transversals of $\mathcal{A}$

This paper  studies the polytopes which arise from lattice
paths. We consider two kinds of lattice paths,
both of which are in the plane. The lattice paths we
consider use steps $E=(1,0)$ and $N=(0,1)$. We will often treat lattice paths as words in the alphabets $\{E,N\}$, and the notation $\alpha^n$ 
denotes the concatenation of $n$ copies
of $\alpha$, where $\alpha$ is a letter or string of letters.

\vspace{5mm}




The lattice path matroids first defined by Bonin et al.~\cite{B1} as follows:

\begin{definition}\label{def:lpm}
Let $P=p_1p_2\cdots p_{m+r}$ and $Q=q_1q_2\cdots q_{m+r}$ be two
lattice paths from $(0,0)$ to $(m,r)$ with $P$ never going above
$Q$. Let $\{p_{u_1},\ldots,p_{u_r}\}$ be the set of North steps of
$P$ with $u_1,u_2,\ldots,u_r$; similarly, let
$\{q_{l_1},\ldots,q_{l_r}\}$ be the set of North steps of $Q$ with
$l_1,l_2,\ldots,l_r$. Let $N_i$ be the interval $[l_i,u_i]$ of
integers. Let $\Mo[P,Q]$ be the transversal matroid that has ground
set $[m+r]$ and presentation $(N_i: i\in [r])$; the pair $(P,Q)$ is
a \emph{presentation of $\Mo[P,Q]$}. A \emph{lattice path matroid} is
a matroid  that is isomorphic to $\Mo[P,Q]$ for some such pair of
lattice paths $P$ and $Q$.
\end{definition}


The fundamental connection between the transversal matroid $\Mo[P,Q]$
and the lattice paths that stay in the region bounded by $P$ and $Q$
is the following theorem of Bonin et. al.~\cite{B1}  which says that the bases of $\Mo[P,Q]$ can
be identified with such lattice paths.
\
\begin{theorem}[Bonin et. al.]\label{bases}
A subset $B$ of $[m+r]$ with $|B|=r$ is a basis of $\Mo[P,Q]$ if and
only if the associated lattice path $P(B)$ stays in the region
bounded by $P$ and ~$Q$, where $P(B)$ is a path which has its North steps on  the set $B$ positions and it has its East steps on the set $[m+r]-B$ positions.
\end{theorem}

A special class of the lattice path matroids are the 
generalized Catalan matroids defined as follows:

\begin{definition}
A lattice path matroid $\Mo$ is a \emph{generalized Catalan matroid}
if there is a presentation $(P,Q)$ of $\Mo$ with $P=E^m N^r$. In this
case we simplify the notation $\Mo[P,Q]$ to $\Mo[Q]$. If in addition the
upper path $Q$ is $(E^k N^l)^n$ for some positive integers $k,l,$
and $n$, we say that $\Mo[(E^k N^l)^n, E^m N^r]$ is the \emph{$(k,l)$-Catalan matroid}
$\Mo_n^{k,l}$.  In place of ${\Mo}_n^{k,1}$ we write ${\Mo}_n^k$; such
matroids are called \emph{$k$-Catalan matroids}.  In turn, we
simplify the notation $\Mo_n^1$ to $\Mo_n$; such matroids are called
\emph{Catalan matroids}.
\end{definition}

The generalized Catalan matroids were discovered by Crapo and rediscovered in various contexts; 
they have been called shifted matroids , PI-matroids ~\cite{Lou}, and freedom matroids.
\

Throughout this paper we investigate  lattice path matroid polytopes.

\section{Faces and Dimensions of Lattice Path Matroid Polytopes.}

In this section, we study the faces and dimensions of the lattice path matroid polytopes.
In general the faces of matroid polytopes are not well understood. 
The following is the main fundamental result in this area.
Edmonds ~\cite{E} as well as Gel$'$fand, Goresky, MacPherson and Serganova~\cite[Thm 4.1]{G}
 show the following characterization of matroid polytopes.

Let $\Mo$ be a matroid, then we have the following:
 \begin{lemma}\label{matedge}
 Two vertices $ e_{B_1}$ and $e_{B_2}$ are adjacent in $\Po(\Mo)$ if and only if $ e_{B_1} -  e_{B_2} =  e_i -  e_j$ for some $i$ and $j$.
\end{lemma}

The circuit exchange axiom gives rise to the following equivalence
relation on the ground set $[n]$ of the matroid $\Mo$: We say $i$ and $j$ are
\textit{equivalent} if there exists a circuit ~$C$ with $\{i,j\}
\subseteq C$. The equivalence classes are the {\em connected
components} of $\Mo$. Let $c(\Mo)$ denote the number of connected
components of $\Mo$. We say that $\Mo$ is connected if $c(\Mo) =1 $. The
following proposition has been shown in \cite{Fstum} by Feichtner and Sturmfels.

\begin{proposition}[Feichtner, Sturmfels]\label{dimformula}
The dimension of the matroid polytope $\Po(\Mo)$ equals $n-c(\Mo)$.
\end{proposition}

Let $P=p_1p_2\dots p_{m+r}$ and $Q=q_1q_2\dots q_{m+r}$ be two
lattice paths from $(0,0)$ to $(m,r)$ with $P$ never going above
$Q$.  The following result explain the number of connected components in the lattice path matroid polytopes.

\begin{proposition}[Bonin \emph{et al.}]\label{dsum}
The class of lattice path matroids is closed under the direct sums.
Furthermore, the lattice path matroid $\Mo[P,Q]$ is connected if and
only if the bounding lattice paths $P$ and $Q$ meet only at $(0,0)$
and $(m,r)$.
\end{proposition}
Applying Propositions ~\ref{dimformula} and ~\ref{dsum}, we have the following lemma:

\begin{lemma} The dimension of the lattice path matroid polytope
$\Po(\Mo[P,Q])$ is $m+r-k+2$, where $k$ is the number of intersection
vertices of the paths $P$ and $Q$.
\end{lemma}

\begin{corollary}\label{dim}
The Catalan matroid polytope $\Po(\Mo_n)$, for any $n\geq 2$, has
$c(\Mo_n)=3$ connected components and its dimension is $2n-3.$
\end{corollary}

In the following  lemma,  we give a combinatorial interpretation of the number of edges of the generalized Catalan matroid polytopes as follows:

\begin{lemma} Consider the lattice path matroid polytope $\Po(\Mo[E^m
N^{r},Q])=\Po(\Mo[Q])$. The number of edges of this polytope is equal to the sum of areas
between the paths from $(0,0)$ to $(m,r)$ which do not go above $Q$ or below the path $E^m
N^{r}.$
\end{lemma}

\begin{proof}
We know that the vertices of the generalized Catalan matroid polytope 
\\
$\Po(\Mo[E^m N^{r},Q])$ are the paths with $m$ East  steps and $r$ North steps which does not exceed $Q$. By Lemma ~\ref{matedge},
the number of edges of this polytope is equal to the number of paths
$P$ and $P^{'}$ in this region which  differ in one $N$
step and one $E$ step consecutively. Without lose of generality,  we may assume that $P=P_1NP_2EP_3$ and $P^{'}=P_1EP_2NP_3.$

For each path $P$ in $[E^mN^{r},Q]$, we can always switch ordered pairs of $N$
step and $E$ step to one other pair of $E$ step and $N$ step and
obtain the other path $P^{'}$ in $[E^mN^{r},Q]$. Clearly, the vertices associated $P$ and $P^{'}$ in $\Mo[E^mN^{r},Q]$ are adjacent to each other.
We only need to count the number of all pairs of paths $P$ and $P^{'}$ which only different in $N$ and $E$ steps consequently.  For any  consecutive pair of $N$ and $E$ steps  in the path $P$, we can construct a path $P^{'}$ which is different by $P$ only in those position. For any path $P$, these pairs of $N$ and $E$ steps are in bijection with squares below path $P$. We can conclude that the  number of all pairs of paths $P$ and $P^{'}$ which only different in $N$ and $E$ steps consequently is equal to the sum of the areas between  all the paths in $[E^mN^{r},Q]$  and  the path $E^{m} N^{r}$ consisting of $N$ and $E$ steps.
 \end{proof}

\begin{lemma}\label{wang} The number of edges, $a(n)$, of the Catalan matroid polytope 
\\
 $\Po(\Mo_{n}) =\Po(\Mo[E^{n}N^{n}, (EN)^{n} ])= \Po(\Mo[E^{n-1}N^{n-1}, (NE)^{n-1} ])$, $a(n)$, is the the total area below paths consisting of  $E$, $(1,0)$, and $N$, $(0,1)$,  steps from $(0,0)$ to $(n,n)$, that stay weakly below
$y=x$. So we can calculate $a(n)$ as follows:
\begin{equation}
a(n) = \frac{n^2}{2}\frac{1}{n+1}\binom{2n}{n}-  \frac{4^n}{2}-\frac{1}{4}\binom{2n+2}{n+1}.
\end{equation}
\end{lemma}
\begin{proof}
 
Let $a(n)$ denote the total area below paths consisting of steps $E$ and $N$ from $(0,0)$ to $(n, n)$ that stay weakly below
$y=x.$ Furthermore, let $A_n$ be the total area between the paths consisting of steps $E$ and $N$ from $(0,0)$ to $(n,n)$ and the line $x=y.$
 The $n$th Catalan number, $C_n$, is the number of paths from $(0,0)$ to $(n, n)$ that stay weakly below
$y=x.$

We proceed by induction, it is not hard to verify that

\begin{equation}\label{no}
A_{n+1}=2\sum_{k=0}^{n}(k+\frac{1}{2}) C_kC_{n-k} +\sum_{k=0}^{n} A_k C_{n-k} + \sum_{k=0}^{n} A_{n-k}C_k.
\end{equation}

Therefore, we have:
\begin{equation}
A_{n+1}=2\sum_{k=0}^{n} A_k C_{n-k}+ \frac{1}{2}\sum_{k=0}^{n} C_k C_{n-k} +\sum_{k=0}^{n} kC_kC_{n-k}.
\end{equation} 

 Let $C(t)$ and $A(t)$ be the generating functions for $C_n$ and $A_n,$ we have
 \begin{equation}
 \frac{A(t)}{t}= 2A(t)C(t)+ \frac{1}{2}{C(t)}^2 + t C^{'}(t)C(t).
 \end{equation}
 
 By differentiating, we obtain the following generating function for $A(t)$ 

\begin{equation}
 A(t)=\frac{{1-2t-\sqrt{1-4t}}}{4t(1-4t)} .
 \end{equation}
 
Therefore, we obtain the value for $A_n$ as follows:

\begin{equation}
 A_n= \frac{4^n}{2}-\frac{1}{4}\binom{2n+2}{n+1} .
\end{equation}

From the definition of $A_n$ and $a(n)$, we have,
\begin{equation}
a(n)= \frac{n^2}{2}\frac{1}{n+1}\binom{2n}{n}-  \frac{4^n}{2}-\frac{1}{4}\binom{2n+2}{n+1}.
 \end{equation}
 
 The number of edges of  the  Catalan matroid polytope  $\Po(\Mo_n)$ is  $a(n)$. 
  \end{proof}


Consider the connected lattice path matroid polytope $\Po({\Mo[P,Q]})$, where $P$ and $Q$ are paths from $(0,0)$ to $(m,r)$. We have $P= E^{{\alpha}_1}N^{{\alpha}_2}\cdots N^{{\alpha}_{2k}}$
and also $Q=N^{{\beta}_1}E^{{\beta}_2}\cdots E^{{\beta}_{2r}}$. 
As we know, any bases of the matroid $\Mo[P,Q]$  associated to the vector $X=x_1\cdots x_{m+r}$, where vector $X$ is a base for $\Mo[P,Q]$ if and only if $P(X)$ lies in the region $[P,Q].$  Let $p_i$ and $q_i$ be the number of $N$ steps which occur in the first $i$ steps of paths $P$ and $Q$, where $1 \leq i \leq m+r$, so $p_{m+r}=q_{m+r}=m$. Therefore, $P(X)$ lies in the region $[P,Q]$ if and only if $p_i \leq x_1+\cdots+x_i \leq q_i$ for all $1\leq i \leq m+r$.

\begin{lemma}\label{hyperplane}
The polytope $\Po({\Mo[P,Q]})$  can be determined by the following inequalities, 
\begin{itemize}
\item $p_i \leq x_1+\cdots+x_i \leq q_i$ for all $1\leq i \leq m+r$,
where $x_1+\cdots+x_{m+r}=m$,
\item $0\leq x_i \leq 1$.
\end{itemize}
where $p_i$ and $q_i$ be the number of $N$ steps that occur in the first $i$ steps of paths $P$ and $Q$.
\end{lemma}

\begin{proof}
Every vertex of $\Po({\Mo[P,Q]})$ satisfy the conditions $(1)$ and $(2)$. Therefore, every point in $\Po({\Mo[P,Q]})$ satisfy both of these conditions.
Now we would like to show that every point $a=(a_1\cdots a_{m+r})$ satisfying inequalities $(1)$ and $(2)$ is 
inside $\Po({\Mo[P,Q]})$. In case  there exists $1\leq i \leq m+r$ so that $a_i=0$, we can proceed by induction on $m+r$. In this case, the point $a$ is in convex hull of the vertices in $\Po({\Mo[P,Q]})$ whose $i$th vertices are $0$, so it lies in $\Po({\Mo[P,Q]})$. Similarly, we can proceed for the case $a_i=1$  Otherwise, let $a_{i}$ be the minimum value of vector in $a$ and let $X_i$ be a vertex whose $i$th coordinate is $1$. We define the vector $B=\dfrac{a-a_iX_i}{1 - a_i}$. This vector satisfies the inequalities conditions and it has zero entries.  By the previous case, the point $B$ lies inside the polytope and so the point $a$. Therefore, we can proceed with induction.

\end{proof}

\begin{lemma}\label{facets}
 Consider the connected lattice path matroid polytope $\Po({\Mo[P,Q]})$, where $P$ and $Q$ are paths from $(0,0)$ to $(m,r)$ so that
$P= E^{{\alpha}_1}N^{{\alpha}_2}\cdots N^{{\alpha}_{2l}}$ and also $Q=N^{{\beta}_1}E^{{\beta}_2}\cdots E^{{\beta}_{2s }}$.  The affine hull of this polytope is $x_1+\cdots+ x_{m+r}=r.$  

We have the following facets:
\begin{enumerate}
\item[$(a)$]  $x_1+\cdots+x_{{\beta}_1+\cdots+ {\beta}_{2k}}    \leq  {\beta}_1+ {\beta}_3+\cdots+{\beta}_{2k-1}$ for $1\leq k< s$.
\item[$(b)$] $x_1+\cdots+x_{{\alpha}_1+\cdots+ \alpha_{2k}}\geq {\alpha}_2+ {\alpha}_4+\cdots+{\alpha}_{2k}$  for $1 \leq k \leq l-1.$

\end{enumerate}

 In case ${\beta}_1>1$ and ${\alpha}_{2l}>1$, the facets in the affine hull $x_1+\cdots+ x_{m+r}=r$ can  be described as follows:

  \begin{enumerate} 
   \item $x_i\geq 0$  for $i=1,\ldots, m+r$, except for $i$'s so that we have the facets $x_1+\cdots+ x_i \geq j$ and $x_1+\cdots+x_{i-1}\leq j$ in $(a)$ and $(b)$ descriptions.

\item  $x_i\leq 1$ for $i=1, \ldots, m+r$,  except for $i$'s so that we have the facets $x_1+\cdots+ x_i \geq j$ and $x_1+\cdots+x_{i+1}\leq j+1$ in $(a)$ and $(b)$ descriptions.  
 \end{enumerate}
 In case ${\beta}_1=1$ and ${\alpha}_{2l}>1$, the facets in the affine hull $x_1+\cdots+ x_{m+r}=r$ can be described as follows:
  \begin{enumerate} 

\item $x_i\geq 0$ for all $i=1,\ldots, m+r$ except for $i$'s so that we have the facets $x_1+\cdots+ x_i \geq j$ and $x_1+\cdots+x_{i-1}\leq j$ in $(a)$ and $(b)$ descriptions.
 \item $x_i\leq 1$ for  all $i=1, \ldots, m+r$, except  $i\leq  1+ {\beta}_2$ and also for $i$'s so that we have the facets $x_1+\cdots+ x_i \geq j$ and $x_1+\cdots+x_{i+1}\leq j+1$ in $(a)$ and $(b)$ descriptions. 
 
  \end{enumerate}

 In case ${\beta}_1=1$ and ${\alpha}_{2l}=1$ the facets in the affine hull  can be described as follows:
 
  \begin{enumerate} 
\item $x_i\geq 0$ for $i=1,\ldots, m+r$, except for $i$'s so that both facets $x_1+\cdots+ x_i \geq j$ and $x_1+\cdots+x_{i-1}\leq j$ in the above descriptions $(a)$ and $(b)$.
 \item $x_i\leq 1$ for $i=1, \ldots, m+r$, except  $i\leq  1+ {\beta}_2$ and  for $i\geq  {\alpha}_1+\cdots+ {\alpha_{2l-2}} $.  and also for $i$'s so that we have the facets $x_1+\cdots+ x_i \geq j$ and $x_1+\cdots+x_{i+1}\leq j+1,$ in the above descriptions $(a)$ and $(b)$. 
  
  \end{enumerate}
  
   In case ${\beta}_1>1$ and ${\alpha}_{2l}=1$ the facets in the affine hull $x_1+\cdots+ x_{m+r}=r$  can be described as follows,

  \begin{enumerate} 
  
\item $x_i\geq 0$  for all $i=1,\ldots, m+r$ except for $i$'s so that both facets $x_1+\cdots+ x_i \geq j$ and $x_1+\cdots+x_{i-1}\leq j$ in the above descriptions $(a)$ and $(b).$
\item  $x_i\leq 1$ for all $i=1, \ldots, m+r$ except for $i$'s so that both facets $x_1+\cdots+ x_i \geq j$ and $x_1+\cdots+x_{i+1}\leq j+1$ in the above descriptions $(a)$ and $(b)$,
 and also for $i\geq   {\alpha}_1+\cdots+ {\alpha_{2l-2}} $.
 
 \end{enumerate}  
  \end{lemma}
  \begin{proof}
  
    Let us recall the fact that each polytope  is the intersection of a finite family of half spaces in its affine hull. The minimal such family determines the facets of polytope. The polytope  $\Po({\Mo[P,Q]})$ lies on the affine hull $x_1+\cdots+x_{m+r}=r$. So we only need to verify that  the polytope $\Po({\Mo[P,Q]})$ obtained by the described half spaces in the affine hull $x_1+\cdots+x_{m+r}=r$  and this set is minimal.
As we described in Lemma~\ref{hyperplane}, the polytope  $\Po({\Mo[P,Q]})$ can be described  as the intersection of the following hyperplanes,
\begin{enumerate}
\item $0 \leq x_i$ and $x_i \leq 1$ for $1 \leq i \leq m+r$,
\item  $x_1+\cdots+x_{{\beta}_1+\cdots+ {\beta}_{2k}+t}    \leq  {\beta}_1+ {\beta}_3+\cdots+{\beta}_{2k-1} +t$ for $1\leq k < s$ and $t \leq {\beta}_{2k+1}$, 
\item  $x_1+\cdots+x_{{\beta}_1+\cdots+ {\beta}_{2k-1}+t}   \leq  {\beta}_1+ {\beta}_3+\cdots+{\beta}_{2k-1}$ for $1\leq k \leq s$, where $t \leq {\beta}_{2k}$,


\item  $x_1+\cdots+x_{{\alpha}_1+\cdots+ {\alpha}_{{2k}-1}+t}\geq {\alpha}_2+ {\alpha}_4+\cdots+{\alpha}_{2k-2}+t$, where $0\leq t \leq {\alpha}_{2k}$ and $1\leq k \leq l$,
\item  $x_1+\cdots+x_{{\alpha}_1+\cdots+ {\alpha}_{{2k}-2}+t}\geq {\alpha}_2+ {\alpha}_4+\cdots+{\alpha}_{2k-2}$, where $0\leq t \leq {\alpha}_{2k-1}$ and   $1\leq k \leq l$. 
\end{enumerate}

It is easy to verify that the hyperplanes    $x_i\leq 1$ and $x_i \geq 0$ for $i=1,\ldots, m+r$, $x_1+\cdots+x_{{\beta}_1+\cdots+ {\beta}_{2k}}    \leq  {\beta}_1+ {\beta}_3+\cdots+{\beta}_{2k-1}$ for $1\leq k< s$ and $x_1+\cdots+x_{{\alpha}_1+\cdots+ {\alpha_{2k}}}\geq {\alpha}_2+ {\alpha}_4+\cdots+{\alpha}_{2k}$  for $1 \leq k \leq l-1$ generate all the hyperplanes stated above. In addition we have the following facts:
\begin{enumerate}
\item  The hyperplanes $x_1+\cdots+ x_i \geq j$ and $x_1+\cdots+x_{i-1}\leq j$ implies that $x_i\geq 0$, so we can omit  the hyperplane  $x_i\geq 0$ for such $i$'s.
\item  The hyperplanes $x_1+\cdots+ x_{i-1} \geq j$ and $x_1+\cdots+x_{i}\leq j+1$   implies that $x_i\leq 1$, so we can omit  the hyperplane  $x_i\leq 1$ for such $i$'s.
\item   In case $\beta_1=1$, the hyperplame  $x_1+\cdots+x_{1+{\beta}_2}\leq 1$ implies that $x_i\leq 1$ for $i\leq \beta_2+1.$
\item   In case $\alpha_{2l}=1$, the equality $x_1+\cdots+x_{m+r}=r$  implies that  $x_i\leq 1$ for $i\geq {\alpha}_1+\cdots+ {\alpha_{2l-2}} $.
\end{enumerate}
 
 Therefore, the hyperplanes stated on the theorem generate $\Po(\Mo[P,Q])$ in the affine hull $x_1+\cdots+x_{m+r}=r.$ It is not hard to verify that the generating set is minimal and none of hyperplanes generate by others.
\end{proof}
  
\begin{lemma}\label{catfaces}
The Catalan matroid polytope   
 $\Po(\Mo_{n+1})=\Po(\Mo[E^{n}N^{n}, (NE)^{n} ])$, for any $n\geq 2$, has
$5n-5$ facets which lies in the following hyperplanes:
 
\begin{itemize}
\item $x_3,\ldots, x_{2n-1}, x_{2n}\leq 1$,
\item $x_1,x_2,\ldots,x_{2n}\geq 0$,
\item $\sum_{i=1}^{2k-2}{x_{i}}\leq k-1$ for $2\leq k  \leq n$,
\end{itemize}
\end{lemma}
  
 \begin{lemma}\label{kcatfaces}
The Generalized Catalan matroid polytope  

 $\Po(\Mo^{r}_n)=\Po(E^{r(n-1)}N^{n-1}, (NE^r)^{n-1})$, for any $n\geq 2$ has $(r+1)(2n-3)+n-2$ facets which lies in the following hyperplanes:

\begin{itemize}
\item $x_{r+2},\ldots, x_{(r+1)(n-1)}\leq 1$,
\item $x_1,x_2,\ldots,x_{(r+1)(n-1)}\geq 0$,
\item $\sum_{i=1}^{k(r+1)}{x_{i}}\leq k$ for $1\leq k  \leq n-2$.
\end{itemize}
\end{lemma} 
  
   


  
  \begin{theorem}\label{faces}
All the faces of lattice path matroid polytope are lattice path matroid polytopes.
\end{theorem}

\begin{proof}
 Without lost of generality, we may assume that the lattice path matroid polytope $\Po({\Mo[P,Q]})$ is connected, where $P$ and $Q$ are paths from $(0,0)$ to $(m,r)$, so that $P= E^{{\alpha}_1}N^{{\alpha}_2}\cdots N^{{\alpha}_{2k}}$ and also $Q=N^{{\beta}_1}E^{{\beta}_2}\cdots E^{{\beta}_{2s}}$.

We wish to show that  all the facets of this polytope are the lattice path matroid polytopes.
Clearly, the vertices in facet of the form $x_1+\cdots+x_i\leq q_i $, $x_1+\cdots+x_i\geq p_i$ are the paths which go through the $i$th vertex of the paths $Q$ and $P$, respectively. Thus, these facets are the lattice path matroid polytopes which are direct sums of two other lattice path matroid polytopes. We only need to show that facets which obtain by equalities $x_i=0$ and $x_i=1$ are also the  lattice path matroid polytopes.

Consider the facet $x_i=1$ of the polytope. The vertices of this facet are paths with $N$ step on their $i$th step. We just delete $i$th element  of the matroid $\Mo[P,Q]$ and as discussed in~\cite{B2} the resulting matroid is a lattice path matroid.  See in Figure~\ref{htp}.  
 The result is lattice path matroid associated to this facet. 
Similarly, the vertices with $x_i=0$, form a lattice path matroid polytope.
\end{proof}

\begin{figure}\label{facet}
\centering
\includegraphics[height=2cm]{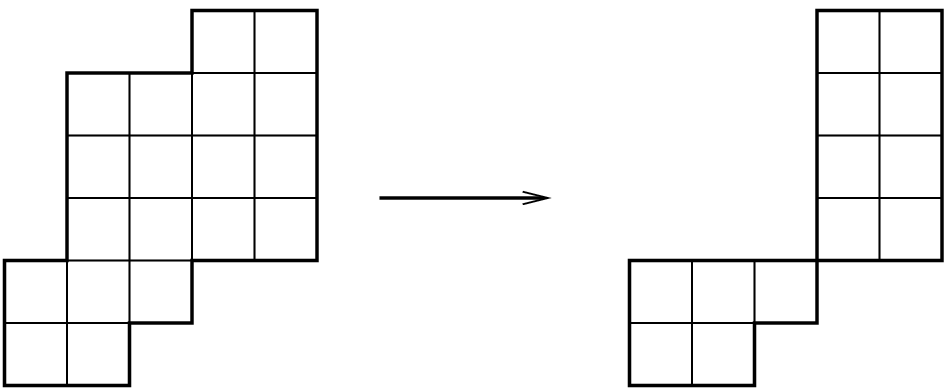}
\caption{Faces of lattice path matroid polytope}
\end{figure}

\begin{figure}\label{htp}
\centering
\includegraphics[height=3.5cm]{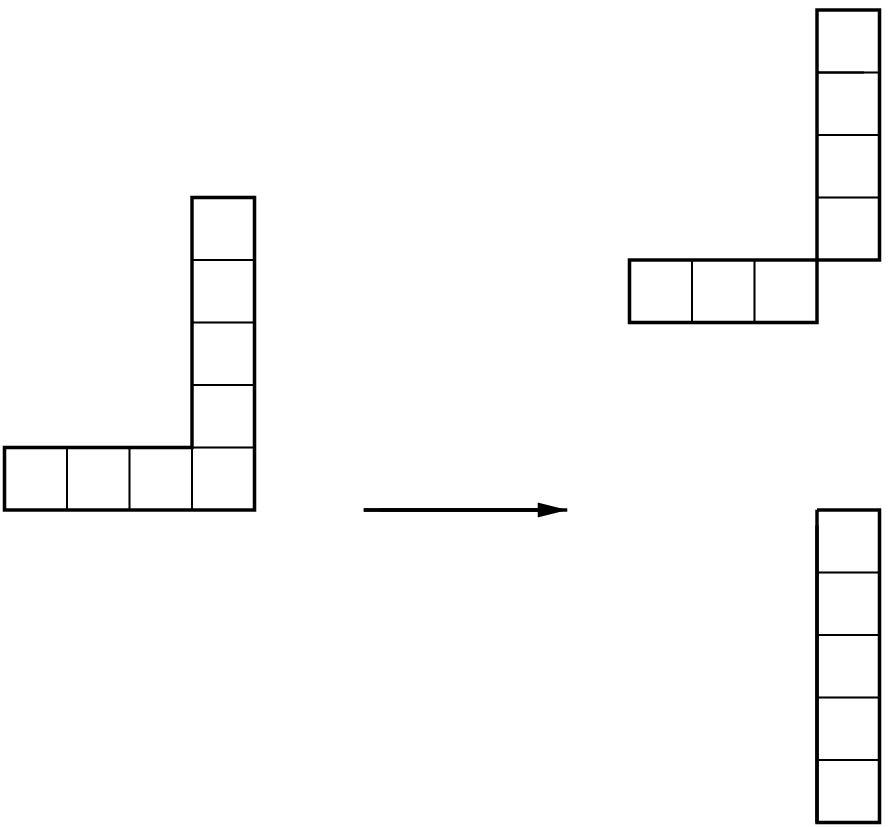} 
\caption{Faces of lattice path matroid polytope}
\end{figure}

\section{Decomposition of Lattice Path Polytope}

 In this section, we study the decomposition of the lattice path matroid polytope into  lattice path matroid polytopes.
 
 Billera, Jia and Reiner~\cite{BJ} defined a matroid polytope decomposition of $\Po(\Mo)$ to
be a decomposition $\Po(\Mo) = \bigcup_{i=1}^{t} \Po(\Mo_i)$ where each
$\Po(\Mo_i)$ is also a matroid  polytope for some matroid $\Mo_i$ and all $\Po(\Mo_i)$'s have the same dimension as $\Po(\Mo)$.
and for each $1\leq i\neq j\leq t$, the intersection $\Po(\Mo_i)\cap
\Po(\Mo_j)$ is a face of both $\Po(\Mo_i)$ and $\Po(\Mo_j )$. The polytope $\Po(\Mo)$ is
said to be decomposable if it has a matroid polytope
decomposition for $t\geq 2$, and it is  indecomposable otherwise.

 A
decomposition is called hyperplane split if $t = 2$. We notice that if $\Po(\Mo) = \Po(\Mo_1)\cup
\Po(\Mo_2)$ is a nontrivial hyperplane split then $\Po(\Mo_1)\cap
\Po(\Mo_2)$ must be a facet of both $\Po(\Mo_1)$ and $\Po(\Mo_2)$, and the
dimension of $\Po(\Mo_i)$ for $i = 1, 2$ is the same as that of $\Po
(\Mo)$.

Let $\Mo =(\mathcal{E}, \mathcal{B})$ be a matroid of rank $r$ and let $A \subseteq \mathcal{E}$. We
recall that the independent set of the restriction matroid of $\Mo$ to
$A$, denoted by $\Mo|_A$, is given by $\I(\Mo|_A) = \{I\subseteq A :
I\in \I(\Mo)\}$. Let $(\mathcal{E}_1, \mathcal{E}_2)$ be a partition of $\mathcal{E}$, that is,
$\mathcal{E}=\mathcal{E}_1\cup \mathcal{E}_2$ and $\mathcal{E}_1\cap \mathcal{E}_2 = \emptyset$. Let $r_i > 1$, $i = 1, 2$
be the rank of $\Mo|_{\mathcal{E}_i}$. We say that $(\mathcal{E}_1,\mathcal{E}_2)$ is a \emph{good
partition} if there exist integers $0 < a_1 < r_1$ and $0 < a_2 <
r_2$ with the following properties:

\begin{itemize}
\item[$(P1)$] $r_1 + r_2 = r + a_1 + a_2$
\item[$(P2)$] For all $X \in \I(\Mo|_{\mathcal{E}_1})$ with
$|X| \leq r_1 - a_1$ and all $Y \in \I(\Mo|_{\mathcal{E}_2})$ with $|Y|  \leq  r_2
- a_2$, we have $X \cup Y \in \I(\Mo)$.
\end{itemize}

\begin{lemma}[Alfonsin, Chatelain]\label{good}
Let $\Mo =(\mathcal{E}, \mathcal{B})$ be a matroid of rank $r$ and let $(\mathcal{E}_1,\mathcal{E}_2)$ be a
good partition of E. Let $\B(\Mo_1) = \{B\in \B(\Mo) : |B \cap \mathcal{E}_1|\leq
r_1 - a_1\}$  and $\B(\Mo_2)= \{B\in \B(\Mo) : |B\cap \mathcal{E}_2|\leq r_2 -
a_2\}$. Where $r_i$ is the rank of matroid $\Mo|_{\mathcal{E}_i}$ for $i = 1, 2$
and $a_1$ and $a_2$ are integers verifying properties $(P1)$ and $(P2)$.
Then $\B(\Mo_1)$ and $\B(\Mo_2)$ are the collections of bases of
matroids $\Mo_1$ and $\Mo_2$, respectively. As a conclusion, $\Po(\Mo)=\Po(\Mo_1)\cup \Po(\Mo_2)$ is a 
hyperplane split.
\end{lemma}


In the following lemma we use the work of Alfonsin and Chatelain to explain the lattice path matroid polytope decompositions.
\begin{lemma}\label{transdecom}
Let $\Mo[P,Q]$ be a lattice path matroid the transversal matroid on $\{1,\ldots,m+ r\}$ and
presentation $(N_i : i \in \{1,\ldots,r\})$ where $N_i$ denotes the
interval $[s_i, t_i]$ of integers. Suppose that there exists integer
$x$ such that $s_j < x < t_j$ and $s_{j+1} < x + 1 < t_{j+1}$ for
some $1\leq j \leq r - 1$. Then, $\Po(\Mo[P,Q])$ has a nontrivial
hyperplane split. In fact this decompose the lattice path matroid polytope  $\Po(\Mo[P,Q])$ into two lattice path matroid polytopes  $\Po(\Mo[P,Q_1])$ and  $\Po(\Mo[P_1,Q)$ \end{lemma}

\begin{figure}[htp1]
\centering
\includegraphics[height=5cm]{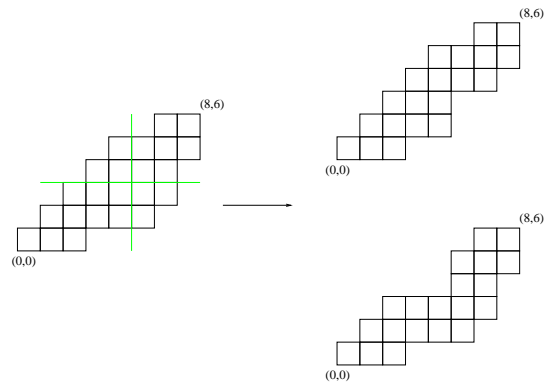}
\caption{Decompositions of matroid polytopes}
\end{figure}

\begin{proof}
We can consider lattice path matroid as a transversal matroid on
$\{1,...,m+ r\}$ and presentation $(N_i : i \in \{1,\cdots,r\})$ where
$N_i$ denotes the interval $[s_i, t_i]$ of integers, were $s_i$ and
$t_i$ are the location of $i$th $N$ on paths $P$ and $Q$
respectively. Since the region of $\Mo[P,Q]$ is not a border strip and
this is not a border strip lattice path matroid, there exists integer $x$
such that $s_j < x < t_j$ and $s_{j+1} < x + 1 < t_{j+1}$ for some
$1\leq j \leq r - 1$.  Set $\mathcal{E}_1 = \{1,
\cdots, x\}$ and $\mathcal{E}_2 = \{x+1,\cdots ,m+r\}$. The partition
$(\mathcal{E}_1,\mathcal{E}_2)$ verifies property $(P1)$ by taking integers $a_1$ and
$a_2$ such that $r_1-a_1= $ and $r_2-a_2=r-j$. Moreover, the sets
$\B(\Mo_1)= \{\B \in \mathcal{\B}(\Mo):|\B \cap \mathcal{E}_1|\leq r_1-a_1\}$ and
$\B(\Mo_2)= \{B \in {\B}(\Mo) :|\B \cap \mathcal{E}_2|\leq r_2-a_2\}$ are the
collections of bases of matroids $\Mo_1$ and $\Mo_2$ respectively.
Indeed, $\Mo_1$ is the transversal matroid with representation $({
\overline{N_i}}^1:i\in \{1, \cdots , r\})$ where ${
\overline{N_i}}^1= N_i$ for each $i = 1, \cdots, j$ and ${
\overline{N_i}}^1 = N_i\cap \mathcal{E}_2$ for each $i = j+1, \cdots, r$.
$\Mo_2$ is the transversal matroid with representation
$({\overline{N_i}}^2: i \in \{1,\cdots, r\})$ where ${
\overline{N_i}}^2 = N_i\cap \mathcal{E}_1$ for each $i = 1,\cdots,j$ and ${
\overline{N_i}}^1 = N_i$ for each $i = j + 1,\cdots, r$. Finally,
$\Mo_1 \cap \Mo_2$ is the transversal matroid with representation $(
\overline{N_i} : i\in \{1,\cdots, r\})$ where $\overline{N_i} = {
\overline{N_i}}^1 \cap { \overline{N_i}}^2 $ for each $i = 1,\cdots
, r$.

Let us consider $\Mo_1$ and $\Mo_2$ as lattice path matroid polytopes.

Consider the point $(h+1-j,j)$ on the region of lattice path matroid. It is not hard to see that $\Mo_1$ and $\Mo_2$ are the lattice
path matroids. The region of $\Mo_1$ is obtained by removing boxes of
$\Mo[P,Q]$ which are above the lines $y=j$ and on the left hand side
of the line $x=h+1-j$. The region of $\Mo_2$ is also obtained by
removing the boxes which are on the right hand side of the vertical
line $x=x-j$ and below $y=j.$

Let $k$ be the least positive integer so that $t_k -k\geq x-j$. Let
$P_1$ be the path of length $m+r$ with the following set of $r$ $N$
steps: $\{t_1, \ldots,t_j, m,m+1, \ldots,m+k-j, t_k, \ldots,t_r \}$. Let
be a path $Q_1$ of length $m+r$ with $r$ north steps $\{s_1, \ldots,
s_l, m-l+k, \ldots,m, s_{k+1}, \ldots, s_r \}$ where $l$ is the greatest
element so that $s_{l}-l\leq m-j$. It is easy to see that
$\Mo_1=\Mo[P_1,Q]$ and $\Mo_2=\Mo[P,Q_1]$. See Figure~\ref{htp1}.

\end{proof}

\begin{definition}
Let $\Mo[P,Q]$ be a connected lattice path matroid so that the region
between $P$ and $Q$ is a connected border strip and let $p$ be a path
whose vertices are boxes of border strip and its edges are connected
boxes. We call $\Mo[P,Q]$  a \emph{border strip matroid} and we denote it
by $S(p).$
\end{definition}

\begin{figure}\label{border}
\centering
\includegraphics[height=3.5cm]{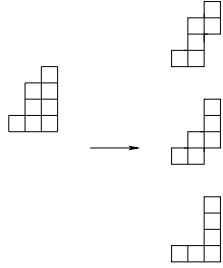}
\caption{Decompositions of matroid polytope to border strips}
\end{figure}

\begin{lemma}\label{dec}
Let $\Po(\Mo[P,Q])$ be a  connected lattice path matroid of rank $r$ on
$\{1,\ldots,m+ r\}$  which is not a border strip matroid. It can be
decomposed into connected lattice path matroid polytope using
hyperplane split. Moreover, $\Po(\Mo[P,Q])$ can be decomposed into $\Po(S(p))$ where $p$ ranges over all paths contained in
$\Mo[P,Q]$.
\end{lemma}

\begin{proof}
By induction we know that $\Po(\Mo[P_1,Q])$,$\Po(\Mo[P,Q_1])$ can be decomposed to
border strip matroid polytopes $\Po(S(p))$ for all 

$p\in [P_1,Q]$ and $\Po(S(p))$ for all $p\in [P_1,Q]$ , respectively.

Since all the paths contain in region $[P,Q]$ is either in their
region $[P,Q_1]$ or in $[P_1, Q]$. Therefore $\Po(\Mo[P,Q])$ can
decompose to border strip matroid polytopes $\Po(S(p))$ for
all $p\in \Mo[P,Q]$
\end{proof}

\section{Triangulation and Ehrhart Series of Catalan Matroid Polytope}

The {\it hypersimplex\/} $\H_{k,n}\subset\R^n$ is the convex
polytope defined as the convex hull of the points $\epsilon_I$, for
$I\in\binom{[n]}{k}$. All these $\binom{n}{k}$ points are actually
vertices of the hypersimplex because they are obtained from each
other by permutations of the coordinates. This $(n-1)$-dimensional
polytope can also be defined as
$$\H_{k,n} = \left\{(x_1,\ldots, x_n)\mid 0\leq x_1,\ldots, x_n\leq 1;\
x_1+\cdots + x_{n} =k \right\}.$$
\

The following unimodular triangulation of hypersimplex introduced by Sturmfels.

\
\subsection{Another Combinatorial Interpretation  of  the Volume of the Lattice Path Matroid Polytopes}

We consider the integers $0< k < n$.  We set $[n]:=\{1,\dots,n\}$,
$\binom{[n]}{k}$
denotes the collection of $k$-element subsets of $[n]$.

Clearly for each $k$-subset $I \in \binom{[n]}{k}$
we  cab associate the $0, 1$-vector $e_I =
(e_1,\dots,e_n)$ such that $e_i=1$ for $i\in I$;
and $\epsilon_i = 0$ for $i\not\in I$.

The {\it hypersimplex\/} $\H_{k,n}\subset\R^n$ is a convex polytope
defined as the convex hull of the points $\epsilon_I$,
for $I\in\binom{[n]}{k}$.
All these $\binom{n}{k}$ points are actually vertices
of the hypersimplex because they are obtained from each other
by permutations of the coordinates.
This $(n-1)$-dimensional polytope can also be defined as
$$
\H_{k,n} = \{(x_1,\dots,x_n)\mid 0\leq x_1,\dots,x_n\leq 1;\
x_1+\cdots + x_{n} =k\}.
$$
The hypersimplex is linearly equivalent to the
polytope $\tilde\Delta_{k,n}\subset\R^{n-1}$ given by
\[
\tilde\Delta_{k,n}=\{(x_1,\dots,x_{n-1})\mid 0\leq x_1,\dots,x_{n-1}\leq 1;\
k-1 \leq x_1 + \cdots + x_{n-1} \leq k\}.
\]
Indeed, the projection $p:(x_1,\dots,x_n)\mapsto (x_1,\dots,x_{n-1})$
sends $\Delta_{k,n}$ to $\tilde\Delta_{k,n}$.
The hypersimplex $\tilde\Delta_{k,n}$ can be thought of as
the region (slice) of the unit hypercube $[0,1]^{n-1}$
contained between the two hyperplanes $\sum x_i = k-1$ and $\sum x_i = k$.

Recall that a {\it descent\/} in a permutation $w\in S_n$ is an
index $i\in\{1,\dots,n-1\}$ such that $w(i)>w(i+1)$. Let $\des(w)$
denote the number of descents in $w$. The {\it Eulerian number\/}
$A_{k,n}$ is the number of permutations in $S_n$ with $\des(w)=k-1$.

Let us normalize the volume form in $\R^{n-1}$ so that the volume of
a unit simplex is 1 and, thus, the volume of a unit hypercube is
$(n-1)!$. It is a classical result, implicit in the work of 
Laplace
that the normalized volume of the hypersimplex $\Delta_{k,n}$ equals
the Eulerian number $A_{k,n-1}$. One would like to present a
triangulation of $\Delta_{k,n}$ into $A_{k,n-1}$ unit simplices.
Such a triangulation into unit simplices is called a {\it unimodular
triangulation}.

In this section we discuss Stanley's triangulation of hypersimplex as follows:


\subsection{Stanley's triangulation} \label{sec:Sta}

The hypercube $[0,1]^{n-1}\subset\R^{n-1}$ can be triangulated into
$(n-1)$-dimensional unit simplices $\nabla_w$ labelled by
permutations $w \in S_{n-1}$ given by
\[
\nabla_w = \set{(y_1,\ldots,y_{n-1}) \in [0,1]^{n-1} \mid 0 <
y_{w(1)} < y_{w(2)} < \cdots < y_{w(n-1)} < 1}.
\]

Stanley~\cite{Sta} defined a transformation of the hypercube
$\psi:[0,1]^{n-1} \rightarrow [0,1]^{n-1}$ by
$\psi(x_1,\ldots,x_{n-1}) = (y_1,\ldots,y_{n-1})$, where
\[
y_i = (x_1 + x_2 + \cdots + x_i) - \lfloor x_1 + x_2 + \cdots +
x_i \rfloor.
\]
The notation $\lfloor x \rfloor$ denotes the integer part of $x$.
The map $\psi$ is piecewise-linear, bijective on the hypercube
(except for a subset of measure zero), and volume preserving.

Since the inverse map $\psi^{-1}$ is linear and injective when
restricted to the open simplices $\nabla_w$, it transforms the
triangulation of the hypercube given by $\nabla_w$'s into another
triangulation.

\begin{theorem}[Stanley~\cite{Sta}]
The collection of simplices $\psi^{-1}(\nabla_w)$, $w\in S_{n-1}$,
gives a triangulation of the hypercube $[0,1]^{n-1}$ compatible
with the subdivision of the hypercube into hypersimplices.
The collection of the simplices $\psi^{-1}(\nabla_w)$,
where $w^{-1}$ varies over permutations in $S_{n-1}$
with $k-1$ descents, gives a triangulation of the $k$-th
hypersimplex $\tilde \Delta_{k,n}$.
Thus the normalized volume of $\tilde\Delta_{k,n}$
equals to the Eulerian number $A_{k,n-1}$.
\label{thm:Sta}
\end{theorem}

\begin{definition} The standard Young tableau of the shape  $\lambda$, where  $|\lambda|=n$ is a filling of $\Lambda$ with the numbers
$1, \cdots, n$ which is increasing in the rows. and decreasing in the columns.

\end{definition}

We know the following facts:

\begin{remark}

\begin{itemize}

\item  The number of Standard Young tableaux of the shape $\lambda$ is $f_{\lambda}$, which can calculated by hook length formula.~\cite{}
\item The number of border strip Young Tableaux's of the shape $S(p)$ is denoted by $f_{S(p)}$. It is not hard to see that $f_{S(p)}$ is exactly the number of permutations of $1,\cdots,n$ which has descents when the step from $i$ to $i+1$ is a horizontal step. 
\end{itemize}
\end{remark}

Using Stanley's triangulation, we show another combinatorial interpretation of the volume of polytope.

\begin{lemma}\label{border}

The volume of the border strip matroid polytope $\Po(S(p))$ is  $f_{S(p)}$ which is the number of standard young tableaus of  the shape $S(p)$.
\end{lemma}

\begin{proof}
Consider a border strip shape Young tableaux of $\lambda,$ where $|\lambda|=n.$ As we discussed before, the standard young tableaux's of the shape $\lambda=S(p)$ is in bijection with permutations of size $n$ which have  a descent on $i$th positions, when there is a box $i+1$ above box $i$ on the border strip young tableaux of the shape $\lambda=S(p).$

Considering Stanley's triangulation, 
recall that this  triangulation
occurs in the space $\R^{n-1}$. To be more precise, in order to
obtain Stanley's triangulation we need to apply the projection
$p:(x_1,\ldots,x_n) \mapsto (x_1,\ldots,x_{n-1}).$  Let us identify a permutation $w=w_1\cdots
w_{n-1}\in S_{n-1}$ with $k-1$ descents with the permutation
$w_1\cdots w_{n-1} n\in S_n$.

Recall the map
$\psi^{-1}:(y_1,\dots,y_{n-1})\mapsto(x_1,\dots,x_{n-1})$
restricted to the simplex $\nabla_w = \set{0 < y_{w(1)} < \cdots <
y_{w(n-1)} <1}$ is given by $x_1 = y_1$ and
$$
x_{i+1} = \left\{
\begin{array}{cl}
y_{i+1}-y_i & \text{if } w^{-1}(i+1)>w^{-1}(i),\\
y_{i+1}-y_i+1 & \text{if }w^{-1}(i+1)<w^{-1}(i)
\end{array}
\right.
$$

So the image of the map $\psi^{-1}$ for $\nabla_w$ lies in the polytope  described by the following hyperplanes
$des(w^{-1})_i \leq x_1+\cdots+x_i  \leq des(w^{-1})_i+1.$ Applying Lemma ~\ref{hyperplane} each border strip tableaux of shape
 $\lambda$ can be described by the hyperplanes $des(w^{-1})_i \leq x_1+\cdots+x_i \leq des(w^{-1})_i+1.$ It is not hard to see that  all the simplexes $\nabla_w $, where $w^{-1}$ have  the same descent sets are map to their associated border strip matroid polytope $P_{S(p)}$, where the horizontal steps of  $p$ are the same as  the descent sets of $w^{-1}$ . This map is injective on the interior of simplexes and  the interior of the border strip matroid polytope is covered by them.

Therefore, the volume of $\Po(S(p))$ is the number of permutations which has descent on the step $i$ if and only if the box $i+1$ is above the box $i.$ This number is equal to the number of Standard young tableaux's of
the shape $S(p).$
\end{proof}

\begin{theorem}
Let $\Po(\Mo[P,Q])$ be a  connected lattice path matroid of rank $r$ on
$\{1,\ldots,m+ r\}$  which is not a border strip matroid. The volume of $\Po(\Mo[P,Q])$  is sum over 
$f_{(S(p))}$ where $p$ is range over all paths contained in
$\Mo[P,Q]$  and $f_{(S(p))}$ is the number of Standard young tableaux of shape $S(p)$ .

\end{theorem}

\begin{proof}

As we know by Lemma~\ref{dec},  $\Po(\Mo[P,Q])$  can be
decomposed into the connected lattice path matroid polytope using
hyperplane splits. Moreover, $\Po(\Mo[P,Q])$ can be decomposed into $\Po(S(p))$ where $p$ is range over all paths contained in
$[P,Q]$. By Lemma~\ref{border}, the volume of $\Po(S(p))$ is the number of Standard young tableaux's  of  shape $S(p).$ 
\end{proof}


\
\

\subsection{ Formula for Ehrhart Polynomial and Volume of Lattice Path Matroid polytopes}

\
Consider the lattice path matroid polytope $\Po({\Mo[P,Q]})$ where $P$ and $Q$ are the paths from $(0,0)$ to $(m,r)$. Let $p_i$ and $q_i$ be the number of $N$ steps occur in the first $i$ steps of paths $P$ and $Q$, respectively, where $1 \leq i \leq m+r$, clearly, $p_{m+r}=q_{m+r}=r$. We know that  $\Po(X)$ lies in the region $[P,Q]$ if and only if $p_i \leq x_1+\cdots+x_i \leq q_i$ for all $1\leq i \leq m+r$. Therefore,     
the polytope $\Po({\Mo[P, Q]})$    can be determined by the following inequalities:

\begin{enumerate}
\item $p_i \leq x_1+\cdots+x_i \leq q_i$ for all $1\leq i \leq m+r,$
where $x_1+\cdots+x_{m+r}=r,$
\item $0\leq x_i \leq 1.$
\end{enumerate}

Let us denote $x_1, x_1+x_2, \ldots, x_1+\cdots+x_{m+r}$ by $c_1, \ldots, c_{m+r}$ so  it is an increasing sequence where $p_i\leq c_i \leq q_i$ and $c_{m+r}=r$.

Consider $a_1,\ldots, a_{r-1}$ and $b_1, \ldots, b_{r-1}$ so that $a_k+1=\min\{ i,  p_i\geq k+1 \}$
and $b_k+1= \min \{i,  q_i\geq k+1\}.$ 
We define the set of arrays of  positive integers $\Gamma(P,Q)$ as follows: 

\begin{enumerate}
\item ${\alpha}_1+\cdots+{\alpha}_{r}=m+r,$  

\item  $a_i \leq {\alpha}_1+\cdots+{\alpha}_i \leq b_i$ for  $i\leq r-1,$ 
\item ${\alpha}_i\geq 1.$
\end{enumerate}

Consider the point $\textbf{x}=(x_1,\ldots, x_{m+r})$ and 
$c_i= x_1+\cdots +x_i$  for $i=1,\cdots, m+1.$
It is easy to verify that the integer point $\textbf{x}=(x_1,\ldots,x_{m+r})$ is in $\Po({\Mo[P,Q]})$ if and only if 
for some $\alpha= ({\alpha}_1, \ldots, {\alpha}_r) \in   \Gamma(P,Q)$ we have:

 \begin{eqnarray}
  \notag
0 \leq c_1 \leq \cdots \leq c_{{\alpha}_1} \leq 1 < c_{{\alpha}_1 +1} \leq \cdots \leq c_{{{\alpha}_1}+{{\alpha}_2}}
\leq 2
\\ <\cdots \leq c_{{{\alpha}_1}+\cdots+ {{\alpha}_{r-1}}} \leq{r-1}< \cdots \leq c_{{{\alpha}_1}+\cdots+ {{\alpha}_{r}}}=r.
\end{eqnarray} 
and $c_{{\alpha}_1}\geq 1.$

We conclude the following theorem.

\begin{theorem}
 The number of lattice points in $\Po(\Mo[P,Q])$ is
$|\Gamma(P,Q)|.$
\end{theorem}

We define the set $S_{r}(t)$ be the set of arrays $(s_1, s_2, \ldots , s_{2r-2})$ so that $s_i + s_{i+1} \leq t$ for all $i=1,\cdots, 2r-3$ as well as $s_1\leq t$ and $s_{2r-2}\leq t.$
Considering the above observations, the  integer points in $t \Po(\Mo[P,Q])$ are in bijection with the following set of sequences.

For any $\alpha \in \Gamma(P,Q)$
\begin{enumerate}\label{c12}

 \item $0\leq c_1\leq c_2 \leq \cdots \leq c_{{\alpha}_1}=t-s_1$, where $c_i -c_{i-1}\leq t$. The number of such sequences is $\left(\binom{t+1-s_1}{{\alpha}_1}\right).$
\item  For $1<i <r$, we have:
\\
 $(i-1)t \leq (i-1)t+ s_{2i-2}= c_{{\alpha}_1+\cdots+{\alpha}_{i-1}+1}\leq \cdots \leq c_{{\alpha}_1+\cdots+{\alpha}_i}= (i)t-s_{2i-1} \leq (i)t$. The number of such sequences are
$\left(\binom{t-s_{2i-1}-s_{2i-2}}{{\alpha}_{i}}\right).$
\item  For $i=r$,  $(r-1)t \leq (r-1)t+ s_{2r-2}= c_{{\alpha}_1+\cdots+{\alpha}_{r-1}+1}\leq \cdots \leq c_{{\alpha}_1+\cdots+{\alpha}_r}=(r)t$. The number of such sequences are
$\left(\binom{t-s_{2r-2}}{{\alpha}_{r}}\right).$
\end{enumerate}

Recall that we define $S_{r}(t)$ as follows:

\begin{eqnarray}
\notag\
S_{r}(t)=\left\{ s=(s_1,\ldots,s_{2(r-1)}) \textnormal{ so that } 
s_1\leq t,  s_1+s_2\leq t, \right.
\\ \left. \ldots, s_{2(r-1)-1}+s_{2(r-1)}\leq t, s_{2(r-1)}\leq t \right\}.
\end{eqnarray}

Observing the above facts, we conclude that  Ehrhart polynomial of the  Lattice Path Matroid Polytope can be computed as follows:

\begin{theorem}
\begin{equation}
\sum_{\alpha \in \Gamma(P,Q)} \sum_{s \in{ S_{r}(t)}}{\left(\binom{t+1-s_1}{{\alpha}_1}\right)}{\left(\binom{t-s_2-s_3}{{\alpha}_2}\right)}\cdots {\left(\binom{t-{s_{2r-2}}}{{\alpha}_r} \right)}.
\end{equation}
\end{theorem}

\section*{Acknowledgements}

I would like to thank Richard Stanley for suggesting the problem and
helpful discussions. I would like to also thank Federico Adrila, Matthias Beck and  Alex Postnikov, Richard Ehrenborg and Seth Sullivant
for helpful discussions. Lemma~\ref{border} is a joint work with Seth Sullivant.
\end{document}